\documentclass[11pt,reqno,oneside]{amsart}

\usepackage[asymmetric,top=3.5cm,bottom=4.3cm,left=3.1cm,right=3.1cm]{geometry}
\usepackage{amsmath,amsfonts,amsthm,mathrsfs,amssymb,cite}
\usepackage[usenames]{color}

\usepackage{mathtools} 
\usepackage{graphics}
\usepackage{framed}

\usepackage{enumerate}

\usepackage{subfigure}

\usepackage{graphicx}
\usepackage{latexsym} 
\usepackage{enumerate}

\theoremstyle{plain}
\newtheorem{theorem}{Theorem}[section]

\newtheorem{lemma}{Lemma}[section]

\newtheorem{definition}{Definition}[section]

\newtheorem{remark}{Remark}[section]

\renewcommand{\eqref}[1]{\textnormal{(\ref{#1})}}

\numberwithin{equation}{section}

\newcommand{\R}{\mathbb{R}}

\def\bsi{{\mathrm{i}}}
\def\Oh{{\mathcal O}}
\newcommand{\bl}{\color{black}}

\begin{document}

\title[Uniqueness in determining a piecewise conductive medium body]{Determining a piecewise conductive medium body by a single far-field measurement}

\author{Xinlin Cao}
\address{Department of Mathematics, Hong Kong Baptist University, Kowloon, Hong Kong, China.}
\email{xlcao.math@foxmail.com}

\author{Huaian Diao}
\address{School of Mathematics and Statistics, Northeast Normal University,
Changchun, Jilin 130024, China.}
\email{hadiao@nenu.edu.cn}

\author{Hongyu Liu}
\address{Department of Mathematics, City University of Hong Kong, Kowloon, Hong Kong, China.}
\email{hongyu.liuip@gmail.com, hongyliu@cityu.edu.hk}


\begin{abstract}

We are concerned with the inverse problem of recovering a conductive medium body. The conductive medium body arises in several applications of practical importance, including the modeling of an electromagnetic object coated with a thin layer of a highly conducting material and the magnetotellurics in geophysics. We consider the determination of the material parameters inside the body as well as on the conductive interface by the associated electromagnetic far-field measurement. Under the transverse-magnetic polarisation, we derive two novel unique identifiability results in determining a 2D piecewise conductive medium body associated with a polygonal-nest or a polygonal-cell geometry by a single active or passive far-field measurement. 

\medskip



\end{abstract}

\maketitle

\section{Introduction}\label{sec1}

\subsection{Physical motivation and mathematical formulation}

We are concerned with the time-harmonic electromagnetic wave scattering from a conductive medium body. The conductive medium body arises in several applications of practical importance, including the modelling of an electromagnetic object coated with a thin layer of a highly conducting material and the magnetotellurics in geophysics. In order to well motivate the current study, we next provide brief discussions on the aforementioned two specific applications and then introduce the mathematical formulation of the associated inverse problems.  

In what follows, the optical properties of a medium are specified the electric permittivity $\varepsilon$, the magnetic permeability $\mu$ and the conductivity $\sigma$. Let $\Omega$ be a bounded Lipschitz domain in $\mathbb{R}^2$ with a connected complement $\mathbb{R}^2\backslash\overline{\Omega}$. Consider a infinitely long cylinder-like medium body $D:=\Omega\times\mathbb{R}$ in $\mathbb{R}^3$ with the cross section being $\Omega$ along the $x_3$-axis for $\mathbf{x}=(x_j)_{j=1}^3\in D$. In what follows, with a bit abuse of notation, we shall also use $\mathbf{x}=(x_1, x_2)$ in the 2D case, which should be clear from the context. Let $\delta\in\mathbb{R}_+$ be sufficiently small and $\Omega_\delta:=\{\mathbf{x}+h\nu(\mathbf{x}); \mathbf{x}\in\partial\Omega\ \mbox{and}\ h\in (0, \delta)\}$, where $\nu\in\mathbb{S}^1$ signifies the exterior unit normal vector to $\partial\Omega$. Set $D_\delta=\Omega_\delta\times\mathbb{R}$ to denote a layer of thickness $\delta$ coated on the medium body $D$. The material configuration associated with the above medium structure is given as follows:
\begin{equation}\label{eq:m1}
\varepsilon, \mu, \sigma=\varepsilon_1, \mu_0, \sigma_1\ \mbox{in}\ \ D;\ \varepsilon_2, \mu_0, \frac{\gamma}{\delta}\ \ \mbox{in}\ D_\delta;\ \varepsilon_0, \mu_0, 0\ \ \mbox{in}\ \mathbb{R}^3\backslash\overline{(D\cup D_\delta)},
\end{equation}
where for simplicity, $\varepsilon_j$, $j=0,1,2$, $\mu_0, \gamma$ are all positive constants and $\sigma_1$ is a nonnegative constant. Consider a time-harmonic incidence: 
\begin{equation}\label{eq:inc1}
\nabla\wedge\mathbf{E}^i-\mathrm{i}\omega\mu_0\mathbf{H}^i=0,\quad \nabla\wedge\mathbf{H}^i+\mathrm{i}\omega\varepsilon_0\mathbf{E}^i=0\quad\mbox{in}\ \mathbb{R}^3,
\end{equation}
where $\mathrm{i}:=\sqrt{-1}$, $\mathbf{E}^i$ and $\mathbf{H}^i$ are respectively the electric and magnetic fields and $\omega\in\mathbb{R}_+$ is the angular frequency. The impingement of the incident field $(\mathbf{E}^i, \mathbf{H}^i)$ on the medium body described in \eqref{eq:m1} generates the electromagnetic scattering, which is governed by the following Maxwell system:
\begin{equation}\label{eq:sca1}
\begin{cases}
& \nabla\wedge\mathbf{E}-\mathrm{i}\omega\mu\mathbf{H}=0,\quad \nabla\wedge\mathbf{H}+\mathrm{i}\omega\varepsilon\mathbf{E}=\sigma\mathbf{E}, \quad\mbox{in}\ \mathbb{R}^3, \\[5pt]
& \mathbf{E}=\mathbf{E}^i+\mathbf{E}^s,\quad\mathbf{H}=\mathbf{H}^i+\mathbf{H}^s,\hspace*{2.55cm}\mbox{in}\ \mathbb{R}^3, \\[5pt]
& \displaystyle{\lim_{r\rightarrow\infty} r\left(\mathbf{H}^s\wedge\hat{\mathbf{x}}-\mathbf{E}^s \right)=0, }\hspace*{3.5cm} r:=|\mathbf{x}|, \hat{\mathbf{x}}:=\mathbf{x}/|\mathbf{x}|,
\end{cases} 
\end{equation}
where as usual one needs to impose the standard transmission conditions, namely the tangential components of the electric field $\mathbf{E}$ and the magnetic field $\mathbf{H}$ are continuous across the material interfaces $\partial D$ and $\partial D_\delta$. The last limit in \eqref{eq:sca1} is known as the Silver-M\"uller radiation condition. 

Under the transverse-magnetic (TM) polarisation, namely,
\begin{align}\notag
\mathbf{E}^i&=\begin{bmatrix}
0\\0\\ u^i(x_1,x_2)
\end{bmatrix},\  \mathbf{H}^i=\begin{bmatrix}
H_1(x_1,x_2)\\
H_2(x_1,x_2)\\
0
\end{bmatrix}, 
\end{align}
and
\begin{align}
\mathbf{E}&=\begin{bmatrix}
0\\0\\ u(x_1,x_2)
\end{bmatrix},\  \mathbf{H}=\begin{bmatrix}
H_1(x_1,x_2)\\
H_2(x_1,x_2)\\
0
\end{bmatrix}, \label{eq:p2}
\end{align}
it is rigorously verified in \cite{BonT} that as $\delta\rightarrow +0$, one has
	\begin{equation}\label{eq:model1}
	\begin{cases}
	\Delta u^-+k^2 q u^-=0 & \mbox{ in }\ \Omega, \\[5pt] 
	\Delta u^+ +k^2  u^+=0 & \mbox{ in }\ \R^2 \backslash \Omega, \\[5pt] 
	u^+= u^-,\quad \partial_\nu u^+ + \lambda u^+=\partial_\nu u^- & \mbox{ on }\ \partial \Omega, \\[5pt]
	u^+=u^i+u^s & \mbox{ in }\ \R^2 \backslash \Omega, \\[5pt] 
	\lim\limits_{r \rightarrow \infty} r^{1/2} \left( \partial_r u^s-\bsi k u^s \right)=0, & \ r=|\mathbf x|,
	\end{cases}
	\end{equation}
where 
\begin{equation}\label{eq:model12}
\mbox{$u^-=u|_{\Omega}$, $u^+=u|_{\mathbb{R}^2\backslash\overline{\Omega}}$\ \ and\ \ $k=\omega\sqrt{\varepsilon_0\mu_0}$, $\lambda=\mathrm{i}\omega\gamma\mu_0$}.
\end{equation}
The last limit in \eqref{eq:model1} is known as the Sommerfeld radiation condition. According to \eqref{eq:m1}, as $\delta\rightarrow+0$, it is clear that the conductivity in the thin layer $D_\delta$ goes to infinity, or equivalently, its resistivity goes to zero. This in general would lead to the so-called perfectly electric conducting (PEC) boundary, which prevents the electric field from penetrating inside the medium body and instead generates a certain boundary current. However, it is noted in our case that the thickness of the coating layer $D_\delta$ also goes to zero, and this allows the electromagnetic waves to penetrate inside the medium body. Nevertheless, the thin highly-conducting layer effectively produces a transmission  boundary condition across the material interface $\partial D$ involving a conductive parameter $\lambda$, which is referred to as the conductive transmission condition. As is known, a perfect conductor is an idealized material which does not exist in nature, and hence the conductive medium body provides a more realistic means to model the electromagnetic scattering from an object coated with a thin layer of a highly conducting material; see \cite{HM, Senior} more relevant discussion about this aspect. 

In fact, the conductive transmission condition has its origin in magnetotellurics and was first introduced by the geophysicists for modelling the physical phenomenon describe above (cf. \cite{Ang}). The electromagnetic induction in magnetotellurics is usually described by the eddy current model as follows (cf. \cite{Amm,Ang}):
\begin{equation}\label{eq:model2}
\begin{cases}
& \nabla\wedge\mathbf{E}=\mathrm{i}\omega\mathbf{B}, \quad \nabla\wedge\mathbf{B}=\mu_0\sigma\mathbf{E}+\mathbf{J}\quad\mbox{in}\ \ \mathbb{R}^3,\\[5pt]
& \mathbf{E}(\mathbf{x})=\mathcal{O}(|\mathbf{x}|^{-1}),\quad \mathbf{B}(\mathbf{x})=\mathcal{O}(|\mathbf{x}|^{-1}), 
\end{cases}
\end{equation}
where $\sigma=\sigma_1\chi_{D}+\sigma_2\chi_{\mathbb{R}^3\backslash\overline{D}}$, and for simplicity both $\sigma_1$ and $\sigma_2$ assumed to be positive constants. In \eqref{eq:model2}, $\mathbf{J}$ signifies a source current, which is assumed to be compactly supported. The following conductive transmission condition across the material interface $\partial D$ was considered in the geophysics literature (cf. \cite{Ang} and the references cited therein):
\begin{equation}\label{eq:model22}
\nu\wedge\mathbf{E}|_+-\nu\wedge\mathbf{E}|_-={\mathbf 0},\quad \nu\wedge \mathbf{B}|_+- \nu\wedge\mathbf{B}|_-=\mu_0\gamma(\nu\wedge\mathbf{E})\wedge \nu  \quad \mbox{on}\ \partial D.
\end{equation}
By introducing
$$
\mathbf{H}(\mathbf x)=\begin{cases}
(\omega /k_1)\mathbf{B}(\mathbf x),\quad \mathbf x \in D,\\[5pt]
	(\omega /k_2)\mathbf{B}(\mathbf x),\quad \mathbf x \in \mathbb{R}^3 \backslash D,
\end{cases}
$$
where $k_j^2=\mathrm{i}\omega\mu_0\sigma_j$ fulfill $\Im k_j>0$, $j=1, 2$, \eqref{eq:model2} and \eqref{eq:model22} can be transformed to
\begin{equation}\label{eq:model3}
	\begin{cases}
	\nabla\wedge\mathbf{E}-\mathrm{i}k_1\mathbf{H}={\mathbf 0},\ \ \nabla\wedge\mathbf{H}+\mathrm{i}k_1\mathbf{E}=(\omega/k_1 )\mathbf{J}  & \mbox{ in }\ D, \\[5pt] 
	\nabla\wedge\mathbf{E}-\mathrm{i}k_2\mathbf{H}={\mathbf 0},\ \ \nabla\wedge\mathbf{H}+\mathrm{i}k_2\mathbf{E}={\mathbf 0} & \mbox{ in }\ \R^3 \backslash D, \\[5pt] 
	\nu\wedge\mathbf{E}|_+-\nu\wedge\mathbf{E}|_-={\mathbf 0}  & \mbox{ on }\ \partial D, \\[5pt]
	\nu\wedge[\nu\wedge(k_2\mathbf{H}|_+-k_1\mathbf{H}|_-)]=\mu_0\gamma\omega\nu\wedge\mathbf{E} & \mbox{ on }\ \partial D, \\[5pt] 
	\mathbf{E}(\mathbf{x})=\mathcal{O}(|\mathbf{x}|^{-1}),\quad \mathbf{H}(\mathbf{x})=\mathcal{O}(|\mathbf{x}|^{-1}), & \mbox{ as } |\mathbf x| \rightarrow +\infty. 
	\end{cases}
	\end{equation}
Next, we consider the transverse-magnetic scattering associated with \eqref{eq:model3} by assuming \eqref{eq:p2} and $\mathbf{J}=[0, 0, \psi(x_1, x_2)]^\top $. By straightforward calculations, one can show that
\begin{equation}\label{eq:model4}
	\begin{cases}
	\Delta u^-+k^2 q  u^-=\widetilde \psi & \mbox{ in }\ \Omega, \\[5pt] 
	\Delta u^+ +k^2 u^+=0 & \mbox{ in }\ \R^2 \backslash \Omega, \\[5pt] 
	u^+= u^-,\quad \partial_\nu u^+ + \lambda u^+=\partial_\nu u^- & \mbox{ on }\ \partial \Omega, \\[5pt]
	u^+(\mathbf{x})=\mathcal{O}(|\mathbf{x}|^{-1}), & \ \mbox{ as } |\mathbf x|\rightarrow +\infty,
	\end{cases}
	\end{equation}
where $u^-=u|_{\Omega}$, $u^+=u|_{\mathbb{R}^2\backslash\overline{\Omega}}$ and
\begin{equation}\label{eq:m2}
\mbox{$k^2=\mathrm{i}\omega\mu_0\sigma_2$ with $\Im k\geq 0$,\ $q=\sigma_1/\sigma_2 $,\ $\widetilde \psi={\mathrm i} \omega \psi $, \ $\lambda={\rm i} \mu_0 \gamma \omega$}.
\end{equation} 

The well-posedness of the wave scattering systemes \eqref{eq:model1} and \eqref{eq:model4} is studied in \cite{BonT,Bon}. There exists a unique solution $u\in H_{loc}^1(\mathbb{R}^2)$ that depends continuously on the input. In particular, the scattered field in \eqref{eq:model1} admits the following asymptotic expansion
\begin{equation}
u^s(\mathbf x)=\frac{e^{\bsi k|\mathbf x|}}{|\mathbf x|^{1/2}}u^{\infty}(\hat{\mathbf x})+\Oh \left(\frac{1}{|\mathbf x|}\right), |\mathbf x|\rightarrow +\infty,
\end{equation}
which holds uniformly in the angular variable $\hat{\mathbf x}=\mathbf x/|\mathbf x| \in {\mathbb S}^{1}$. $u^{\infty}(\hat{\mathbf x})$ is known as the far-field pattern associated with $u^i$, which encodes the information of the scattering object, namely the conductive medium body $(\Omega; q, \lambda)$. By introducing an abstract operator $\mathcal{F}$ via the Helmholtz system \eqref{eq:model1} which sends $(\Omega; q, \lambda)$ to the associated far-field pattern $u^{\infty}(\hat{\mathbf{x}}; u^i)$, the aforementioned inverse problem can be formulated as 
\begin{equation}\label{inverse}
	\mathcal{F}(\Omega; q,\lambda )=u^{\infty}(\hat{\mathbf{x}}; u^i),\quad\hat{\mathbf{x}}\in\mathbb{S}^1. 
\end{equation}
In this paper, we shall mainly consider the case with a single far-field measurement, namely $u^\infty(\hat{\mathbf{x}}; u^i)$, $\hat{\mathbf{x}}\in\mathbb{S}^1$, associated with a single incident wave $u^i$. It is noted that $u^\infty(\hat{\mathbf{x}}; u^i)$ is generated by exerting a known incident wave, and this is known as the active measurement in the inverse scattering theory. There are mainly two motivations for our study with a single far-field measurement. First, the study of inverse scattering problems with a minimal/optimal number of far-field measurements have been a longstanding and intriguing topic in the literature with a colourful history, see e.g. \cite{AR,BL2016,BL2017,CY,Liua,LPRX,LRX,Liu-Zou,Liu-Zou3,Ron2} and the references cited therein, and it would be interesting to investigate whether one can establish similar results in the new setup for the conductive medium body. Second, the inverse problem associated with the scattering model \eqref{eq:model4} can be formulated in a similar manner to \eqref{inverse}:
\begin{equation}\label{inverse2}
\mathcal{F}(\Omega; q, \lambda)=u^\infty(\hat{\mathbf{x}}; \widetilde\psi),\quad\hat{\mathbf{x}}\in\mathbb{S}^1,
\end{equation}
where $u^\infty(\hat{\mathbf{x}}; \widetilde\psi)$ is the associated far-field pattern generated by the (unknown) source $\widetilde\psi$, which is known as the passive measurement in the literature. Clearly, in the latter inverse problem discussed above, it is more practical to consider the case with a single far-field measurement. Finally, we would like to remark that for \eqref{eq:model4}, it would be more practical to consider the so-called near-field measurement (cf. \cite{Amm}). However, in order to unify the exposition, we still such measurement as the far-field pattern, which should be clear in the physical context.

\subsection{Connection to existing results and discussion}

Though in deriving \eqref{eq:model1} and \eqref{eq:model2}, we have assumed that both $q$ and $\lambda$ are constants, it is clear that one has the same formulation for the case with $q\in L^\infty(\Omega)$ and $\lambda\in L^\infty(\partial\Omega)$, being generic variable functions. In such a case, it can be verified that the inverse problems \eqref{inverse} and \eqref{inverse2} are over-determined. In fact, one can count that the cardinality of $u^{\infty}(\hat{\mathbf{x}}; u^i)/u^\infty(\hat{\mathbf{x}}; \widetilde\psi)$ associated with a fixed $u^i/\widetilde\psi$ is $1$, whereas the cardinality of the unknown $(\Omega; q, \lambda)$ is $2$. Here, by the cardinality of a quantity we mean the number of the independent variables in the quantity. Hence, it is unpractical to consider the inverse problem in the generic case. Nevertheless, for the inverse problem \eqref{inverse}, it is proved in \cite{CDL2} that if $\Omega$ is a convex polygon, then $\Omega$ and $\lambda$ can be uniquely determined independent of a generic variable $q$. Compared to our study in \cite{CDL2}, there are two interesting and significant extensions in the current work. First, in addition to the support $\Omega$ and the boundary conductive parameter $\lambda$ of the medium body, we shall further determine the material parameter inside the body, i.e. $q$. Moreover, we develop technical tricks that can handle more general conductive parameter $\lambda$ than that was considered in \cite{CDL2}. Second, we consider a more complicated physical setup than that in \cite{CDL2}. For example, we may allow the presence of multiple conductive interfaces that are embedded in a layered medium, and this shall be more evident in Section 2. 
Our study of the inverse problem \eqref{inverse} is also closely related to \cite{BL2017} by the third author. Roughly speaking, the results derived in \cite{BL2017} are the particular case $\lambda\equiv 0$ of the results derived in the present paper. Hence, in this sense, our study extends the relevant one in \cite{BL2017} with the standard transmission condition across the material interface to the case with a more general conductive transmission condition. However, as shall be seen, the presence of the conductive transmission condition brings several new technical challenges that requires some highly nontrivial treatments. Moreover, according to our earlier discussion, the current study poses some interesting extensions for future investigations, especially those related the full Maxwell system \eqref{eq:sca1} or \eqref{eq:model3} that have a strong background of practical applications. Finally, we give two remarks about the inverse problem \eqref{inverse2}. First, the recovery of a medium body by the associated passive scattering measurement generated by an unknown source have recently been extensively studied in the literature due to its practical significance. In \cite{LU1}, this type of inverse problem was studied in the context of thermoacoustic and photoacoustic tomography, which is an emerging medical imaging modality. Similar inverse problems were also considered in \cite{DLL1,DLL2} associated with the magnetohydrodynamical system and in \cite{DLU} associated with the Maxwell system that are related to geomagnetic anomaly detection and brain imaging, respectively. In two recent articles \cite{LLM1,LLM2}, this type of inverse problem was considered in the quantum scattering governed by a random Schr\"odinger system. According to our earlier discussion, the new inverse problem \eqref{inverse2} may also find potential applications in magnetotellurics, say e.g. in detecting an underwater submarine whose metal surface is obviously thin and highly conducting, and also where the unknown $f$ signifies the ``mysterious" source that produces the geomagnetic field. Second, our analysis in establishing the unique identifiability results for the inverse problems \eqref{inverse} and \eqref{inverse2} is localized around the conductive material interface. The arguments in dealing with the inverse problem \eqref{inverse} can be easily extended to deal with the inverse problem \eqref{inverse2} as long as the support of $\widetilde\psi$ stays a positive distance away from the conductive interface. Hence, in what follows, we shall mainly treat the inverse problem \eqref{inverse} and remark the corresponding extensions to the inverse problem \eqref{inverse2}. 

The rest of the paper is organized as follows. In Section \ref{sec:2}, we present the geometric and mathematical setups of our study. In Section \ref{sect:3}, we state the main unique identifiability results. Section \ref{sect:4}  derives several critical auxiliary lemmas and Section \ref{sec3} is devoted to the proofs of the main theorems. The paper is concluded with some relevant discussion in Section \ref{sect:6}.

\section{Geometric and mathematical setups}\label{sec:2}

We first present two geometric setups for the conductive medium body $\Omega$. They are the so-called polygonal-nest geometry and {polygonal-cell geometry (cf. \cite{BL2017})}. Through out the rest of the paper, we assume that the support of the medium body $\Omega$ is a bounded simply-connected Lipschitz domain with a connected complement $\mathbb{R}^2\backslash\overline{\Omega}$. It is pointed out that the simply-connectedness is not necessary in our subsequent analysis and all the results can be extended to cover the case that $\Omega$ has multiple simply-connected components, but we make this assumption in order to simplify our exposition. 

\begin{definition}\label{def:1}
$\Omega$ is said to have a polygonal-nest partition if there exist $\Sigma_\ell$, $\ell=1, 2, \ldots, N$, $N\in \mathbb{N}$, such that each $\Sigma_\ell$ is an open convex simply-connected polygon and
\begin{equation}\label{eq:def1}
\Sigma_N\Subset\Sigma_{N-1}\Subset\cdots\Subset\Sigma_2\Subset\Sigma_1=\Omega. 
\end{equation}
\end{definition}

Fig.~\ref{fig:c-n} (a) presents a typical polygonal-nest partition of $\Omega$ with three layers.

\begin{definition}\label{def:2}
$\Omega$ is said to have a polygonal-cell partition if there exist $\Sigma_\ell$, $\ell=1, 2, \ldots, N$, $N\in \mathbb{N}$, such that the following conditions are fulfilled:
\begin{enumerate}
\item each $\Sigma_\ell$ is an open simply-connected polygon and $\Sigma_\ell\subset\Omega$, $\Sigma_{\ell}\cap\Sigma_{\ell'}=\emptyset$ if $\ell\neq \ell'$;

\item $\bigcup_{\ell=1}^N\overline{\Sigma_\ell}=\overline{\Omega}$;

\item for each $\Sigma_\ell$, there exists at least one vertex $\mathbf{x}_{c,\ell}$ such that the two edges of $\partial\Sigma_\ell$ associated with $\mathbf{x}_{c,\ell}$, say $\Gamma_\ell^\pm$, satisfy $\Gamma_\ell^\pm\subset\partial\Omega$. 
 \end{enumerate}
\end{definition}

\begin{figure}[htbp]
	\centering
	 \subfigure[Polygonal-nest geometry]{
	\begin{minipage}[t]{0.4\textwidth}
		\centering
		\includegraphics[width=4cm]{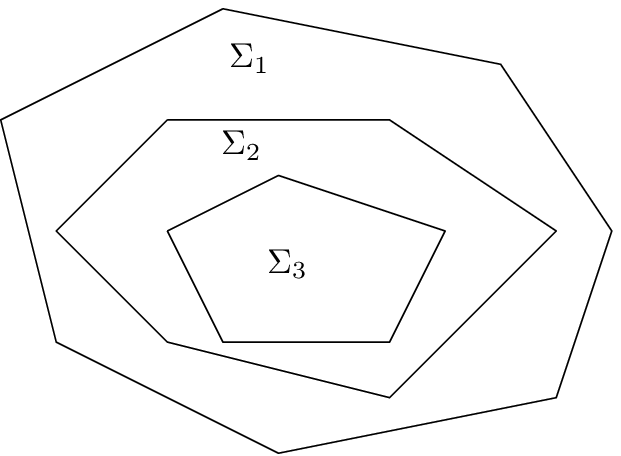}
	\end{minipage}}
	\subfigure[Polygonal-cell geometry]{
	\begin{minipage}[t]{0.4\textwidth}
		\centering
		\includegraphics[width=4cm]{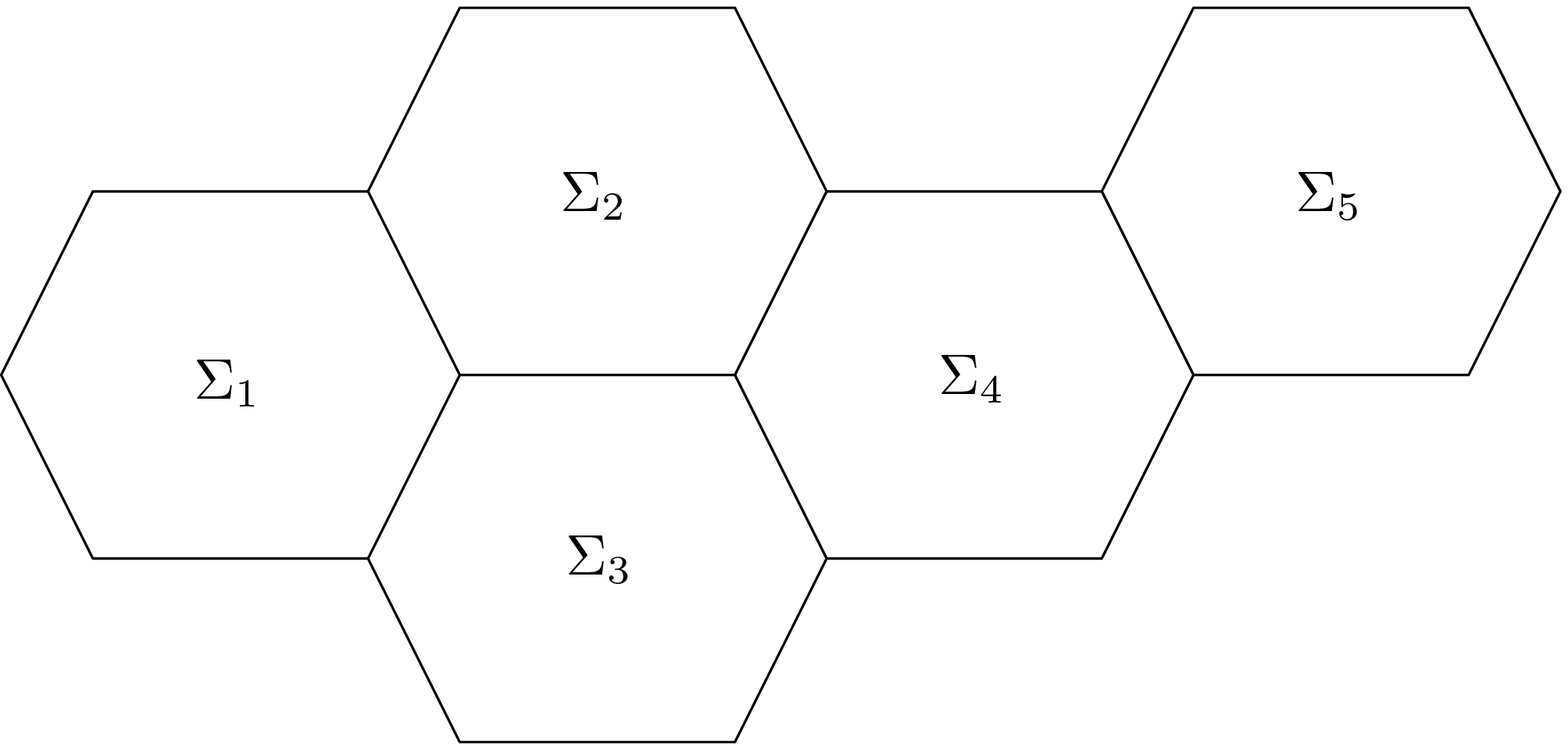}
	\end{minipage}}
  
	\caption{Schematic illustration of the two polygonal geometries in $\R^2$ for a conductive medium body.}
	\label{fig:c-n}
\end{figure}

Fig.~\ref{fig:c-n} (b) presents a typical polygonal-cell partition of $\Omega$ with five hexagonal cells. It is interesting to note that it is the honeycomb graphene structure. We would like to emphasize that for a polygonal-cell partition, each cell is not necessary to be convex. 

\begin{definition}\label{def:3}
Let $(\Omega; q, \lambda)$ be a conductive medium body. It is said to possess a piecewise polygonal-nest structure if the following conditions are fulfilled:
\begin{enumerate}
\item $\Omega$ has a polygonal-nest partition as described in Definition~\ref{def:1};

\item each $\Sigma_\ell$ is a conductive medium body such that {$U_\ell:=\Sigma_\ell\backslash\overline{\Sigma_{\ell+1}}$ possesses material parameters $q_\ell$ and $\lambda_{\ell}$ (on $\partial\Sigma_\ell$), which is denoted as $(U_\ell; q_\ell, \lambda_{\ell})$} for $\ell=1, 2, \ldots, N$ with $\Sigma_{N+1}:=\emptyset$; 

\item for each $(U_\ell; q_\ell, \lambda_{\ell})$, $q_\ell\in \mathbb{C}$ with $\Re q_\ell>0$,  and $\lambda_\ell\in \mathbb{C}$ with $\Re\lambda_\ell\geq 0$ or $\Im\lambda_\ell\geq 0$. 
\end{enumerate}

\end{definition}

For a polygonal-nest conductive medium body $(\Omega; q, \lambda)$ as described in Definition~\ref{def:3}, we write it as
\begin{equation}\label{eq:def3}
(\Omega; q, \lambda)=\bigcup_{\ell=1}^N (U_\ell; q_\ell, \lambda_{\ell})
\end{equation}
and
\begin{equation}\label{eq:def32}
\Omega=\bigcup_{\ell=1}^N U_\ell,\ \ q=\sum_{\ell=1}^N q_\ell \chi_{U_\ell},\ \ \lambda=\sum_{\ell=1}^N \lambda_{\ell} \chi_{\partial\Sigma_\ell}.
\end{equation}
We would like to remark that,  for $\ell\neq \ell'$, $(q_\ell, \lambda_\ell)$ can be different from $(q_{\ell'}, \lambda_{\ell'})$, and $\lambda_\ell$ may be identically zero.

For the wave scattering associated with a a polygonal-nest conductive medium body in \eqref{eq:def3}--\eqref{eq:def32}, it will be described by a PDE system similar to \eqref{eq:model1} or \eqref{eq:model4}, but subject to a proper modification of the conductive transmission conditions in order to accommodate the more general medium structure as follows:
\begin{equation}\label{eq:conductm1}
u_{\ell}|_{\partial\Sigma_{\ell+1}}=u_{\ell+1}|_{\partial\Sigma_{\ell+1}},\quad (\partial_\nu u_{\ell}+\lambda_{{\ell+1}}u_{\ell})|_{\partial\Sigma_{\ell+1}}=\partial_\nu u_{\ell+1}|_{\partial\Sigma_{\ell+1}}, 
\end{equation}
where $u_{\ell}=u|_{U_\ell }$, $\ell=0, 1, 2, \ldots, N-1$ and $\Sigma_0:=\mathbb{R}^2\backslash\overline{\Sigma_1}$. 
The well-posedness of a unique solution $u\in H_{loc}^1(\mathbb{R}^2)$ to the scattering system \eqref{eq:model1} or \eqref{eq:model4} with the conductive transmission condition replaced to be \eqref{eq:conductm1} can be established by following a similar variational argument to that in \cite{Bon}. We shall not explore this issue further since the focus of the present paper is the corresponding inverse problems, i.e., 
\begin{equation}\label{inverse11}
	\mathcal{F}(\bigcup_{\ell=1}^N (U_\ell; q_\ell, \lambda_{\ell}))=u^{\infty}(\hat{\mathbf{x}}; u^i),\quad\hat{\mathbf{x}}\in\mathbb{S}^1. 
\end{equation}
and 
\begin{equation}\label{inverse22}
	\mathcal{F}(\bigcup_{\ell=1}^N (U_\ell; q_\ell, \lambda_{\ell}))=u^{\infty}(\hat{\mathbf{x}}; \widetilde\psi),\quad\hat{\mathbf{x}}\in\mathbb{S}^1. 
\end{equation}

In a similar manner, a piecewise polygonal-cell conductive medium body can be defined as follows. 

\begin{definition}\label{def:4}
Let $(\Omega; q, \lambda)$ be a conductive medium body. It is said to possess a piecewise polygonal-cell structure if the following conditions are fulfilled:
\begin{enumerate}
\item $\Omega$ has a polygonal-cell partition as described in Definition~\ref{def:2};

\item each $\Sigma_\ell$ is a conductive medium body with material parameters $q_\ell$ and $\lambda^*$ and is denoted as $(\Sigma_\ell; q_\ell, \lambda^*)$; 

\item for each $(\Sigma_\ell; q_\ell, \lambda^*)$, $q_\ell\in \mathbb{C}$ with $\Re q_\ell>0$ and $\Im q_\ell\geq 0$, and $\lambda^*\in \mathbb{C}$ with $\Re\lambda^*\geq 0$ or $\Im\lambda^* \geq 0$. 
\end{enumerate}

\end{definition}

Similar to \eqref{eq:def3}--\eqref{eq:def32}, we denote a polygonal-cell conductive medium body $(\Omega; q, \lambda)$ as described in Definition~\ref{def:4} as
\begin{equation}\label{eq:def4}
(\Omega; q, \lambda)=\bigcup_{\ell=1}^N (\Sigma_\ell; q_\ell, \lambda^*)
\end{equation}
and
\begin{equation}\label{eq:def42}
\Omega=\bigcup_{\ell=1}^N \Sigma_\ell,\ \ q=\sum_{\ell=1}^N q_\ell \chi_{\Sigma_\ell},\ \ \lambda=\sum_{\ell=1}^N \lambda^* \chi_{\partial\Sigma_\ell}.
\end{equation}
The direct and inverse problems for a piecewise polygonal-cell conductive medium body can be formulated similar to the case with a piecewise polygonal-nest conductive medium body.

\section{Auxiliary lemmas}\label{sect:3}

In this section, we shall establish several auxiliary lemmas that are critical for our subsequent analysis. We first fix a geometric notation that shall be used through out the rest of the paper. 

 Let $(r,\theta)$ be the polar coordinates in $\R^2$, namely $\mathbf{x}=r(\cos\theta, \sin\theta)$. For $\mathbf{x}\in\R^2$, $B_h(\mathbf{x})$ denotes the open disk of radius $h\in\R_+$ and centred as $\mathbf{x}$. Consider an open sector in $\R^2$ as follows, 
\begin{equation}\label{nota1}
W=\{\mathbf{x}\in\R^2|\mathbf{x}\neq\mathbf{0},\quad \theta_m<\arg(x_1+\bsi x_2)<\theta_M\},
\end{equation}
where $\mathbf{x}=(x_1,x_2)$, $-\pi<\theta_m<\theta_M<\pi$ such that $\theta_M-\theta_m\neq\pi$, $\bsi:=\sqrt{-1}$ and $\Gamma^+$ and $\Gamma^-$ respectively correspond to $(r, \theta_M)$ and $(r, \theta_m)$ with $r>0$. 

Set
\begin{equation}\label{nota2}
	S_h=W\cap B_h,\quad \Gamma_h^\pm=\Gamma^\pm\cap B_h, \quad\overline{S}_h=\overline{W}\cap B_h, \quad \Lambda_h=S_h\cap\partial B_h, 
\end{equation}
where $B_h:=B_h(\mathbf{0})$. In what follows, the sector $S_h$ actually denotes a neighbourhood of a vertex corner of a polygonal conductive medium body. Through out the rest of the paper, we shall always assume that $h\in\mathbb{R}_+$ is sufficiently small such that $S_h$ is completely contained in the concerned conductive medium body, and hence $\Gamma_h^\pm$ lie completely on the two edges associated with the aforesaid vertex, which should be clear from the context. 

We shall also need to make use of the following complex geometric optics (CGO) solution whose logarithm is a branch of the square root stated in the following lemma to prove our main results (cf. \cite{Bsource,CDL2}).
\begin{lemma}\label{lem3-1}\cite[Lemma 2.2]{Bsource}
	For $\mathbf x\in \R^2$ denote $r=|\mathbf x|,\, \theta={\rm arg}(x_1 +\bsi x_2)$. Let
	\begin{equation}\label{eq:u0}
	u_0(\mathbf x):= \exp \left( \sqrt r \left(\cos \left(\frac{\theta}{2}+\pi\right) +\bsi \sin \left(\frac{\theta}{2} +\pi\right)  \right ) \right) .
	\end{equation}
	Then $\Delta u_0=0$ in $\R^2\backslash  \R^2_{0,-} $, where $\R^2_{0,-}:=\{\mathbf{x}\in\R^2|\mathbf{x}=(x_1,x_2); x_1\leq0, x_2=0\}$, and $s \mapsto u_0(s\mathbf x) $ decays exponetially in $\R_+$.  Let $\alpha, s >0$. Then
	\begin{equation}\label{eq:xalpha}
	\int_W |u_0(s\mathbf x)| |\mathbf x|^\alpha {\, d} \mathbf x \leq \frac{2(\theta_M-\theta_m )\Gamma(2\alpha+4) }{ \delta_W^{2\alpha+4}} s^{-\alpha-2},
	\end{equation}
	where $\delta_W=-\max_{ \theta_m < \theta <\theta_M }  \cos(\theta/2+\pi ) >0$. Moreover
	\begin{equation}\label{eq:u0w}
	\int_W u_0(s\mathbf x) {\, d} \mathbf x= 6 \bsi (e^{-2\theta_M \bsi }-e^{-2\theta_m \bsi }  ) s^{-2},
	\end{equation}
	and for $h>0$
	\begin{equation}\label{eq:1.5}
	\int_{W \backslash B_h } |u_0(s\mathbf x)|   {\, d} \mathbf x \leq \frac{6(\theta_M-\theta_m )}{\delta_W^4} s^{-2} e^{-\delta_W \sqrt{hs}/2}. 
	\end{equation}
\end{lemma}



The following lemma is of critical importance for our subsequent analysis of the inverse problems \eqref{inverse11} and \eqref{inverse22}. 

 \begin{lemma}
 	\label{Th:1.2}
Let $S_h$ be defined in \eqref{nota2}. Suppose that $v \in H^2(S_h )$ and $w  \in H^1(S_h) $ satisfy the following PDE system, 
	\begin{align}\label{eq:in eig}
\left\{
\begin{array}{l}
\Delta w+k^2q w=0 \hspace*{1.6cm} \mbox{ in } S_h, \\[5pt] 
\Delta v+ k^2  v=0\hspace*{1.85cm}\ \mbox{ in } S_h \\[5pt] 
w= v,\ \partial_\nu v + \lambda v=\partial_\nu w \ \ \mbox{ on } \Gamma_h^\pm ,
  \end{array}
\right.
\end{align}
where  $\nu\in\mathbb{S}^{1}$ signifies the exterior unit normal vector to $\Gamma_h^\pm $, $k \in \mathbb{C}\backslash\{0\} $, $q\in \mathbb{C}\backslash\{0\} $ and $\lambda(\mathbf x) \in C^{\alpha}(\overline{\Gamma_h^\pm } )$.	It is assumed that either $\lambda\equiv 0$ or $\lambda ({\mathbf 0}) \neq 0$. 
	  Under the following assumptions: 
	\begin{itemize}
	\item[(a)] the function $w$ satisfies a H\"older condition for $\alpha\in (0, 1)$:
	\begin{equation}\label{eq:th21 cond}
			 w  \in C^\alpha(\overline {{S}_h }) ,
	\end{equation}
\item[(b)]  the angles $\theta_m$ and $\theta_M$ of $S_h$ satisfy
	\begin{equation}\label{eq:ass31}
-\pi < \theta_m < \theta_M < \pi \mbox{ and } \theta_M-\theta_m \neq \pi,
\end{equation}
	\end{itemize}
	one has that
	\begin{equation}\label{eq:vn1}
	v({\mathbf 0}) =w({\mathbf 0})=0. 
	\end{equation}
 \end{lemma}
 
 \begin{proof}
 
 For the case with $\lambda\equiv 0$, the proof of the lemma can be found in \cite{BL2017}, and for the case with $\lambda ({\mathbf 0}) \neq 0$, the lemma can be derived by using a similar argument for proving \cite[Theorem 2.2]{CDL2}.  
 \end{proof}
 
 \begin{remark}
 It is emphasized that in both \cite{BL2017} and \cite{CDL2}, $k$ is assumed to be a positive constant. However, the arguments developed therein for deriving the vanishing behaviour \eqref{eq:vn1} also work for the general case with $k\in\mathbb{C}\backslash\{0\}$. We consider such a general $k$ such that for our subsequent analysis can also accommodate the inverse problem \eqref{inverse22} associate with \eqref{eq:model4}, where $k$ is obviously a nonzero complex constant. 
  \end{remark}
 
 \begin{remark}
 Lemma~\ref{Th:1.2} actually summarizes the existing results on the vanishing properties of the transmission eigenfunctions in a rather general setting (cf. \cite{Bsource,BLLW,BL2017,BL2017b,CDL2}), which may be useful for other studies in different contexts. 
 \end{remark}

Lemma~\ref{Th:1.2} is mainly used in proving the determination of the geometric shape of a conductive medium body. The technical requirement $w\in C^\alpha(\overline S_h)$ can be easily fulfilled in the corresponding proof where an contradiction argument is used (see Theorems \ref{Th:3.1} in what follows for the relevant details). However, in order to determine the material parameter inside the medium body, we shall need the H\"older-regularity of the solution $u$ to \eqref{eq:model1} near the vertex corner of the conductive medium body. This can be derived by using the classical results on the singular behaviours of solutions to elliptic PDEs in a corner domain \cite{Dauge88,Grisvard,Cos,CN}. In fact, it is known that the solution can be decomposed into a singular part and a regular part, where the singular part is of a H\"older form that depends on the corner geometry as well as the boundary and right-hand inputs. For our subsequent use, we have the following result in a relatively simple scenario.

\begin{lemma}\label{lem:22}
Suppose that $u \in H^1(S_h)$ is a solution to 
\begin{equation}\label{eq:lem22}
	\left\{\begin{aligned} \Delta u +\omega^{2} u &=f \quad \text { in } S_h, \\  u &=g \quad \text { on } \Lambda_{h}, \\ u &=0 \quad \text { on } \Gamma_h^\pm, \end{aligned}\right.
\end{equation}
where $f\in L^2(S_h)$ is a given source term, $g\in C^{\infty}(\overline{\Lambda_{h}})$ and  $\omega \in \mathbb C $ is a constant. Then there exists the following decomposition 
$$
u=\tilde{u}_{R}+\tilde{c}_{\omega}(f) \chi(r) r^{\alpha} \sin (\alpha \theta),\quad 0<\alpha<1,
$$
where $\tilde{u}_R\in H^2(S_h)$ is the regular part, $\tilde{c}_{\omega}(f) \chi(r) r^{\alpha} \sin (\alpha \theta) \in H^{1+\alpha -\epsilon }(S_h)$ ($\epsilon>0$) represents the singularity of the solution, $\tilde{c}_{\omega}(f)$ is a constant depending on the data of the problem, and $\chi(r)$ is a $C^\infty$ cutoff function that equals 1 in a neighborhood of the origin and 0 close to $\Lambda_h$. 
\end{lemma}

For a convenient reference of the Lemma~\ref{lem:22}, we refer to \cite[Theorem 3.2]{CN}, though it can be found in a more general context in the aforementioned literature. In our case with \eqref{eq:model1} or \eqref{eq:model4}, from the standard PDE theory (see e.g. \cite{McLean}), we know that the solution $u$ is real-analytic away from the conductive interface. Using Lemma~\ref{lem:22}, we can further establish the H\"older-regularity of the solution up to the conductive interface, especially to the vertex corner point. 

\begin{lemma} \label{lem23}
Suppose that $u\in H^1(B_{2h})$ satisfies 
	\begin{equation}\label{eq:lem23}
\begin{cases}
\Delta u^-+k^2 q_- u^-=0 & \mbox{ in }\ S_{2h}, \\[5pt] 
\Delta u^+ +k^2  q_{+} u^+=0 & \mbox{ in }\ B_{2h} \backslash {\overline{ S_{2h}}}, \\[5pt] 
u^+= u^- & \mbox{ on }\ \Gamma_{2h}^\pm, 
\end{cases}
\end{equation}
where $u^+=u|_{{B_{2h}\backslash {\overline{ S_{2h}}}}}$, $u^-=u|_{S_{2h}}$ and $q_{\pm}$, $k$ are complex constants. Assume that $u^+$ and $u^-$ are respectively real analytic in ${B_{2h}\backslash {\overline{ S_{2h}}}}$ and $S_{2h}$. There exists $\alpha\in (0, 1)$ such that $u^- \in C^\alpha(\overline{S_h})$. 
\end{lemma}

\begin{proof}

Since $u^+$ is real analytic in $B_{2h}\backslash\overline{S_{2h}}$, we let $w$ denote the analytic extension of $u^+|_{B_h\backslash\overline{S_h}}$ in $B_h$. By using the transmission condition on $\Gamma_h^\pm$, one clearly has that $u^-=u^+=w$ on $\Gamma_h^\pm$. Set $v=u^- -w$. It can be directly verified that
	\begin{equation}\notag 
\begin{cases}
\Delta v+k^2 q_- v=f & \mbox{ in }\ S_h, \\[5pt] 
v= g & \mbox{ on }\ \Lambda_h, \\[5pt] 
v= 0 & \mbox{ on }\ \Gamma_h^\pm, 
\end{cases}
\end{equation}
where {\bl $f=-(\Delta w+k^2 q_- w)\in L^\infty(S_h)$} and {\bl $g\in C^{\infty}(\overline \Lambda_{h})$}. By virtue of Lemma \ref{lem:22}, it is clear that there exists $\alpha\in (0, 1)$ such that $v\in C^\alpha(\overline{ S_h} )$. Hence, we readily have that $u^-\in C^\alpha ( \overline{ S_h} )$. 

The proof is complete. 
\end{proof}


\section{Statement of main uniqueness results}\label{sect:4}

In establishing the uniqueness results for the inverse problems \eqref{inverse11} and \eqref{inverse22}, we shall require that the total wave fields in \eqref{eq:model1} and \eqref{eq:model4} do not vanish at the vertices of the underlying conductive cells. For a convenient reference in our subsequent arguments, we make it an admissibility condition as follows. 


\begin{definition}\label{admissible}
	Let $(\Omega;q,\lambda )$ be polygonal-nest or polygonal-cell conductive medium body  as described in Definitions~\ref{def:3} and \ref{def:4}, respectively. The scatterer is said to be admissible if it fulfills the following condition: consider a polygonal-nest conductive medium body with the polygonal-nest partition $\{\Sigma_{\ell} \}_{\ell=1}^N$, for any vertex $\mathbf{x}_c\in\partial\Sigma_\ell$, $u(\mathbf x_c)\neq0$; consider a  polygonal-cell conductive medium body, for any vertex $\mathbf{x}_c\in\partial\Omega$, $u(\mathbf x_c)\neq0$, where $u$ is the solution to \eqref{eq:model1} or \eqref{eq:model4}. 
\end{definition}

The admissibility condition in Definition~\ref{admissible} was also used in \cite{CDL2}, which is a non-void condition and can be fulfilled in certain generic scenarios, e.g. for the low-wavenumber case, namely $k$ is sufficiently small. We shall not explore this point in this paper and we refer to \cite[Page 38]{CDL2} for more relevant discussions regarding it. 

We are in a position to present the main unique identifiability results. Our strategy is to first determine the geometric structure of a conductive medium body, namely the underlying polygonal partition and then the conductive parameter and finally the medium parameter inside the body. Our arguments shall be mainly confined around a vertex corner, and hence it applies to both \eqref{inverse11} associated with \eqref{eq:model1} and \eqref{inverse22} associated with \eqref{eq:model4}. As mentioned earlier, we next mainly deal with \eqref{inverse11} associated with \eqref{eq:model1}, and finally remark the extension to \eqref{inverse22} associated with \eqref{eq:model4}. 

We first present a theorem that establishes a ``local'' uniqueness  regarding the shape of an admissible polygonal-nest or polygonal-cell conductive medium body by a single  far-field measurement without knowing the potential $q$ and the conductive parameter $\lambda$. For the polygonal-nest medium body, the ``local" uniqueness results readily implies a ``global" uniqueness result. The proof is similar to the proof of \cite[Theorem 4.1]{CDL2} by using Lemma \ref{Th:1.2}. However, we should emphasize that the technical condition \eqref{eq:th21 cond} in Lemma \ref{Th:1.2} can be automatically fulfilled in proving Theorem \ref{Th:3.1}.   For the completeness, we give the detailed proof of Theorem \ref{Th:3.1} here.

\begin{theorem}\label{Th:3.1}
Consider the conductive scattering problem \eqref{eq:model1} associated with two admissible polygonal-nest or polygonal-cell conductive medium bodies $(\Omega_j; q_j, \lambda_j)$, $j=1,2$, in $\mathbb{R}^2$. Let $u_\infty^j(\hat{\mathbf x}; u^i)$ be the far-field pattern associated with the conductive medium body $(\Omega_j; q_j, \lambda_j)$ and the incident field $u^i$, respectively. Suppose that 
\begin{equation}\label{eq:nn1}
u_\infty^1(\hat{\mathbf x}; u^i)=u_\infty^2(\hat{\mathbf x}; u^i)
\end{equation}
for all $\hat{\mathbf x}\in\mathbb{S}^{1}$ and a fixed incident wave $u^i$. Then 
\begin{equation}\label{eq:nn3}
\Omega_1\Delta\Omega_2:=\big(\Omega_1\backslash\Omega_2\big)\cup \big(\Omega_2\backslash\Omega_1\big)
\end{equation}
cannot possess a corner. Furthermore, if $\Omega_1$ and $\Omega_2$ are two admissible polygonal-nest conductive medium bodies, one must have
\begin{equation}\label{eq:nn4}
\partial \Omega_1=\partial \Omega_2. 
\end{equation}
\end{theorem} 

\begin{proof}
We prove the first part of the theorem by an absurdity argument. By contradiction, we assume that there is a corner contained in $\Omega_1\Delta\Omega_2$. We only consider the case with polygonal-nest case and the polygonal-cell case can be proved in a similar manner. Using the notation \eqref{eq:def3}, we can write the admissible polygonal-nest  conductive medium bodies $(\Omega_j; q_j, \lambda_j)$, $j=1,2$, as follows
	\begin{equation}\label{eq:def3two}
(\Omega_j; q_j, \lambda_j)=\bigcup_{\ell=1}^{N_j} (U_{\ell,j}; q_{\ell,j}, \lambda_{{\ell,j}}),
\end{equation}
and
\begin{equation}\label{eq:def32two}
\Omega_j=\bigcup_{\ell=1}^{N_j} U_{\ell,j},\ \ q_j=\sum_{\ell=1}^{N_j}  q_{\ell,j} \chi_{U_{\ell,j}},\ \ \lambda_j=\sum_{\ell=1}^{N_j} \lambda_{{\ell,j} }\chi_{\partial \Sigma_{\ell,j} }. 
\end{equation}
Without loss of generality we can assume that the vertex $O$ of the corner $\Omega_2 \cap W$ satisfies that $O \in \partial \Omega_2$ and $ O \notin \overline{\Omega}_1$, where $O$ coincides with the origin, and furthermore we assume that the concerned corner satisfies that $S_h\Subset\Sigma_{1,2} $, where $\Sigma_{1,2} $ is defined in \eqref{eq:def3two}.

Let $u_1(\mathbf x)$ and  $u_2(\mathbf  x)$ be the total wave {\bl fields} associated with $\Omega_1$ and $\Omega_2$, respectively. 	Since $u_\infty^1(\hat {\mathbf x}; u^i)=u_\infty^2(\hat {\mathbf x}; u^i)  $ for all $\hat {\mathbf x} \in  {\mathbb S}^1$, applying Rellich's Theorem (see \cite{CK}),  we know that $u_1^s=u_2^s$ in $\R^2 \backslash (  \overline{\Omega}_1 \cup \overline{\Omega}_2 )$. Thus
	\begin{equation}\label{eq:u1u2}
		u_1(\mathbf x)=u_2(\mathbf  x)
	\end{equation}
	for all $\mathbf x \in \R^2 \backslash (  \overline{\Omega}_1 \cup \overline{\Omega}_2 )$. Following the notations in \eqref{nota2}, by using the conductive transmission condition \eqref{eq:conductm1},   we have that
	$$
	u_2^-=u_2^+=u_1^+,\quad \partial_\nu u_2^- = \partial_\nu u_2^+ + \lambda_{{1,2}}  u_2^+=\partial_\nu u_1^+ +  \lambda_{{1,2}}  u_1^+ \mbox{ on } \Gamma_h^\pm,
	$$ 
	where the superscripts $(\cdot)^-, (\cdot)^+$ stand for the limits taken from $\Sigma_{1,2} \Subset \Omega_2$ and $\mathbb{R}^2\backslash\overline{\Omega}_2$ respectively. Moreover, it is supposed the neighbourhood $B_h(O)$ is  sufficiently  small such that
	$$
	\Delta u_1^+ + k^2 u_1^+ =0 \mbox{ in } B_h,\quad \Delta u_2^-  +k^2 q_{1,2} u_2^-=0 \mbox{ in } S_h\Subset \Sigma_{2,1}, 
	$$
where $q_{1,2}$ is defined in \eqref{eq:def32two}.  Clearly $u_1^+\in H^2(S_h)$ and $u_2^-\in H^1(S_h)$.  Due to Lemma \ref{lem23}, we have $ u_2^- \in C^\alpha (\overline{S_h} )$.   Applying Lemma  \ref{Th:1.2} and using the fact that $u_1$ is continuous at the vertex $\mathbf 0$, we have 
	$$
	u_1({\mathbf 0})=0, 
	$$
	which  contradicts to the admissibility condition  in Definition \ref{admissible}.   Therefore we prove that $\Omega_1\Delta\Omega_2$ cannot possess a corner under \eqref{eq:nn1}.  

{\bl Furthermore, if $\Omega_1$ and $\Omega_2$ are two admissible polygonal-nest  conductive medium bodies fulfilling $\partial\Omega_1\neq\partial \Omega_2$}, 	we know that $\Omega_1$ and $\Omega_2$ are convex from Definition  \ref{def:1}. Hence there must exits a corner lying on the boundary of $\R^2 \backslash (  \overline{\Omega}_1 \cup \overline{\Omega}_2 )$. Hence \eqref{eq:nn4} can be easily obtained by using the conclusion in the first part of this theorem.

	The proof is complete. 
\end{proof}

Next we present our main unique recovery {\bl results} for  the piecewise constant scattering potential $q$ and conductive parameter $\lambda$ by a single far-field measurement when the scatterer associated with  \eqref{eq:model1} is an admissible polygonal-nest or polygonal-cell conductive medium body. It is pointed out that we need the a-prior knowledge on the cell structure of an admissible polygonal-cell conductive medium body. However, the unique recovery results for $q$ and $\lambda$ can be established 
for an admissible polygonal-nest conductive medium body by a single measurement without knowing its nest partition. Indeed, the nest partition of an admissible polygonal-nest  conductive medium body can be uniquely determined by a single far-field pattern.   The proofs of the following two theorems are postponed to Section \ref{sec3}.

\begin{theorem}\label{main1}
Considering the conductive scattering problem \eqref{eq:model1} associated with two admissible polygonal-cell  conductive medium  bodies $(\Omega; q_j,\lambda_j)$ in $\R^2$, $j=1,2$.  For $j=1,2$, we let the material parameters $q_j$ and $\lambda_{j }$ with a common polygonal-cell partition $\{\Sigma_{\ell} \}_{\ell=1}^N$ described in Definition \ref{def:2} be characterized by
	\begin{equation}\label{q-m1}
		q_j=\sum_{\ell=1}^{N}q_{\ell,j}\chi_{\Sigma_{\ell,j}},\ \ \lambda_j=\sum_{\ell=1}^{N} \lambda_{j }^*\chi_{\partial \Sigma_{\ell,j} }. 
			\end{equation}  
 Let $u_j^{\infty}(\hat{\mathbf{x}}; u^i)$ be the corresponding far-field pattern associated with the incident wave $u^i$ corresponding to  $(\Omega; q_j,\lambda_j)$, respectively. Suppose that $(\Omega; q_j,\lambda_j)$, $j=1,2$, fulfill
	\begin{equation}\label{main eq1}
	u_1^{\infty}(\hat{\mathbf{x}}; u^i)=u_2^{\infty}(\hat{\mathbf{x}}; u^i)
	\end{equation}
	for all $\hat{\mathbf{x}}\in\mathbb{S}^1$ and a fixed incident wave $u^i$. Then we have $q_{1}=q_{2}$ and $\lambda_{1}=\lambda_{2}$.
\end{theorem}

\begin{theorem}\label{main2}
Considering the conductive scattering problem \eqref{eq:model1} associated with two admissible polygonal-nest conductive medium bodies $(\Omega_j; q_j,\lambda_j)$ in $\R^2$, $j=1,2$, where the associated polygonal-nest partitions $\{\Sigma_{\ell,1}\}_{\ell=1}^{N_1}$ and $\{\Sigma_{\ell,2}\}_{\ell=1}^{N_2}$  are described in Definition \ref{def:1} with $\Omega_1=\bigcup_{\ell=1}^{N_1} U_{\ell,1}$ and $\Omega_2=\bigcup_{\ell=1}^{N_2} U_{\ell,2}$, where $U_{\ell,j}=\Sigma_{\ell,j} \backslash\overline{\Sigma_{\ell+1,j}}$, $\ell=1,\ldots, N_j$.  For $j=1,2$, we let the material parameters $q_j$ and $\lambda_{j }$ be characterized by
	\begin{equation}\label{q-m1}
		q_j=\sum_{\ell=1}^{N_j}q_{\ell,j}\chi_{U_{\ell,j}},\ \ \lambda_j=\sum_{\ell=1}^{N_j} \lambda_{{\ell,j} }\chi_{\partial \Sigma_{\ell,j} }. 
			\end{equation}  
		  Let $u_j^{\infty}(\hat{\mathbf{x}}; u^i)$ be the corresponding far-field pattern associated with the incident wave $u^i$ corresponding to {\bl $(\Omega_j; q_j,\lambda_j)$}, respectively. Suppose that  {\bl $(\Omega_j; q_j,\lambda_j)$}, $j=1,2$, fulfill
	\begin{equation}\label{main eq1}
	u_1^{\infty}(\hat{\mathbf{x}}; u^i)=u_2^{\infty}(\hat{\mathbf{x}}; u^i)
	\end{equation}
	for all $\hat{\mathbf{x}}\in\mathbb{S}^1$ and a fixed incident wave $u^i$. Then we have $N_1=N_2=N$,
	$\partial \Sigma_{\ell,1}= \partial \Sigma_{\ell,2}$ for $\ell=1,\ldots, N$,  $q_1=q_2$ and $\lambda_1=\lambda_2$.

\end{theorem}

\section{Proofs of Theorems \ref{main1} and \ref{main2}}\label{sec3}


In this section, we give the detailed proofs of Theorems \ref{main1} and \ref{main2}. Before that, we first derive the following lemma, which shall be used in our proofs. 

\begin{lemma}\label{prop1}
	Recall that $S_h$, $\Gamma_h^\pm$ and $B_h$ are defined in \eqref{nota2}.  Consider the following PDE system 
	\begin{equation}\label{eq:lemma41}
\begin{cases}
\Delta u^-_1+k^2 \omega_1 u^-_1=0 & \mbox{ in }\ S_h, \\[5pt] 
\Delta u^+_1 +k^2  \omega_1^+ u^+_1=0 & \mbox{ in }\ B_h \backslash \overline{ S_h}, \\[5pt] 
\Delta u^-_2+k^2 \omega_2 u^-_2=0 & \mbox{ in }\ S_h, \\[5pt] 
\Delta u^+_2 +k^2  \omega_2^+ u^+_2=0 & \mbox{ in }\ B_h \backslash \overline{ S_h}, \\[5pt] 
u^+_1= u^-_1,\quad \partial_\nu u^+_1 + \eta_1  u^+_1=\partial_\nu u^- _1& \mbox{ on }\  \Gamma_h^\pm , \\[5pt]
u^+_2= u^-_2,\quad \partial_\nu u^+_2 + \eta_2  u^+_2=\partial_\nu u^- _2& \mbox{ on }\  \Gamma_h^\pm , 
\end{cases}
\end{equation}
where $\omega_j$ and $\omega_j^+$,  $j=1,2$,  are nonzero constants {fulfilling that $\Re\omega_j>0$, $\Im \omega_j\geq0$, $\Re\omega_j^+>0$ and $\Im\omega_j^+\geq0$} in $S_h$ and $B_h\backslash \overline{ S_h}  $, respectively.  In addition,  $\eta_j$,  $j=1,2$, are constants on $\Gamma_h^{\pm }$ {fulfilling that $\Re\eta_j\geq0$ and $\Im\eta_j\geq0$}.  
Assume that $u_j=u_j^-\chi_{S_h}+u_j^+\chi_{B_h\backslash\overline{S_h}}\in H^1(B_h)$, $j=1, 2$, are solutions to \eqref{eq:lemma41} satisfying
	\begin{equation}\label{11}
	u_1^{+}=u_2^{+}\quad \mbox{ in } B_h \backslash \overline{ S_h},\quad \mbox{ and }u_2 ({\mathbf 0}) \neq 0.
	\end{equation}
Then we have {$\eta_1=\eta_2$ and $\omega_1=\omega_2$}. Furthermore, it yields that
\begin{equation}\label{eq:lem41 u1-}
		u_1^{-}=u_2^{-}\quad \mbox{in}\quad  S_h. 
\end{equation}
\end{lemma}

	
\begin{remark}\label{rem:nn1}
It is remarked that the PDE system \eqref{eq:lemma41} actually comes from the scattering system \eqref{eq:model1} or \eqref{eq:model4} that is localized around a corner of the concerned conductive medium body. Here, $\omega_j, \omega_j^+$ and $\eta_j$, are respectively the scattering potentials and the conductive parameters. This shall be clear in our subsequent proofs of Theorems~\ref{main1} and \ref{main2}. Hence, all the lemmas from Section~\ref{sect:3} can be applied in its proof in the following. 
\end{remark}

\begin{proof}[Proof of Lemma~\ref{prop1}] 

We divide the proof into three parts. 

\medskip

	\noindent {\bf Part 1.}~We first prove that $\eta_1=\eta_2$. 
	
	From Lemma \ref{lem23}, we have
\begin{equation}\label{eq:43 calpha}
u^-_1 \in C^\alpha(\overline{S_h}) \mbox{ and } u^-_2 \in C^\alpha(\overline{S_h}). 
\end{equation}
Due to \eqref{11}, by the trace theorem, we know that $u_1^+=u_2^+$ on $\Gamma_h^\pm$, which also implies $\partial_\nu u_1^+=\partial_\nu u_2^+$ on $\Gamma_h^\pm$. Thus from the conductive transmission boundary conditions in \eqref{eq:lemma41}, we can derive that $u_1^-=u_1^+=u_2^+=u_2^-$ on $\Gamma_h^\pm$.

	Define $v:=u_1^--u_2^-$, by direct calculations, we know that it satisfies
	\begin{equation}\label{prop1-2}
	\begin{cases}
	\Delta v+k^2\omega_1v=k^2(\omega_2-\omega_1)u_2^-&\mbox{in }S_h,\\
	v=0, \quad\partial_\nu v=(\eta_1-\eta_2)u_2^-&\mbox{on }\Gamma_h^\pm .
	\end{cases}
	\end{equation}

%
%
	
Recall that the CGO solution $u_0$ is defined in \eqref{eq:u0}. 		Since $u_0\notin H^2(B_\epsilon(\mathbf{0}))$ for $0<\epsilon<h$ around the origin,
define $D_\epsilon:=S_h\backslash B_\epsilon$. Then we can derive the following integral identiy by the fact that $\Delta u_0=0$:
		\begin{align}\label{int-id}
		\int_{D_\epsilon}\Delta v(\mathbf x) u_0(s\mathbf{x})\,d\mathbf x&=\int_{\Lambda_\epsilon}\left(\partial_\nu v(\mathbf x)u_0(s\mathbf{x})-\partial_\nu u_0(s\mathbf{x})v(\mathbf{x})\right)\,d\sigma\notag\\
		&+\int_{\Lambda_h}\left(\partial_\nu v(\mathbf x)u_0(s\mathbf{x})-\partial_\nu u_0(s\mathbf{x})v(\mathbf{x})\right)\,d\sigma\notag \\
		&+\int_{\Gamma^\pm_{(\epsilon,h)}}\left(\partial_\nu v(\mathbf x)u_0(s\mathbf{x})-\partial_\nu u_0(s\mathbf{x})v(\mathbf{x})\right)\,d\sigma, 
		\end{align}
		where $\Lambda_h=S_h \cap \partial B_h$, $\Lambda_\varepsilon  =S_h \cap \partial B_\varepsilon$  and $\Gamma_{(\varepsilon,  h)}^\pm =\Gamma^\pm \cap (B_h \backslash B_\varepsilon )$.
		
		Taking limit as $\epsilon\rightarrow 0^+$, since $|u_0(s\mathbf x)|\leq1$, we have
		\begin{equation}\label{lem21}
		\lim_{\epsilon\rightarrow0}\int_{D_\epsilon}\Delta v(\mathbf x)u_0(s\mathbf{x})\,d\mathbf x=\int_{S_h}\Delta v(\mathbf x)u_0(s\mathbf{x})\,d\mathbf x.
		\end{equation}
Using a similar argument in the proof \cite[Theorem 4.1]{CDL2}	, with the help of  \cite[Theorem 4.18]{McLean}, one can show that $v=u_1^--u_2^-\in H^2(D_\epsilon )$. Therefore we can deduce that
		\begin{equation}\label{lem22}
		\lim_{\epsilon\rightarrow0}\int_{\Lambda_\epsilon}\left(\partial_\nu v(\mathbf x)u_0(s\mathbf{x})-\partial_\nu u_0(s\mathbf{x})v(\mathbf{x})\right)\,d\sigma=0.
		\end{equation}
	Moreover, since $|u_0(s\mathbf{x})|\leq1$ and $u_2^-\in H^1(S_h)$ which indicates that $u_2^-\in L^2(\Gamma_h^\pm)$, by the trace theorem, we have
	\begin{align}\label{b-lim}
	\lim_{\epsilon\rightarrow0}\int_{\Gamma^\pm_{(\epsilon,h)}}\left(\partial_\nu v(\mathbf{x})u_0(s\mathbf x)-\partial_\nu u_0(s\mathbf{x})v(\mathbf{x})\right)\,d\sigma&=\int_{\Gamma_h^\pm}\partial_\nu v(\mathbf{x})u_0(s\mathbf x)\,d\sigma\notag\\
	&=\int_{\Gamma_h^\pm}(\eta_1-\eta_2)u_2^-(\mathbf x)u_0(s\mathbf{x})\,d\sigma,
	\end{align}
	due to the fact that $v=0$ on $\Gamma_h^\pm$.
	Hence we can derive the following integral identity from \eqref{prop1-2} and \eqref{int-id}:
	\begin{align}\notag  
	\int_{S_h}u_0(s\mathbf x)\Delta v\,d\mathbf{x}&=\int_{S_h}\left(k^2(\omega_2-\omega_1)u_2^--k^2\omega_1v\right)u_0(s\mathbf{x})\,d\mathbf{x}\notag\\
	&=\int_{\Gamma_h^\pm}(\eta_1-\eta_2)u_2^-(\mathbf x)u_0(s\mathbf{x})\,d\sigma
	+\int_{\Lambda_h}\left(\partial_\nu v(\mathbf{x})u_0(s\mathbf x)-\partial_\nu u_0(s\mathbf{x})v(\mathbf{x})\right)\,d\sigma, \notag
	\end{align}
	which can be simplified as 
	\begin{align}\label{sim1}
	\int_{S_h}k^2(\omega_2-\omega_1)u_2^-u_0(s\mathbf{x})\,d\mathbf{x}&=\int_{S_h}k^2\omega_1vu_0(s\mathbf{x})\,d\mathbf{x}+\int_{\Gamma_h^\pm}(\eta_1-\eta_2)u_2^-(\mathbf x)u_0(s\mathbf{x})\,d\sigma\notag\\
	&+\int_{\Lambda_h}\left(\partial_\nu v(\mathbf{x})u_0(s\mathbf x)-\partial_\nu u_0(s\mathbf{x})v(\mathbf{x})\right)\,d\sigma,
	\end{align}
	by rearranging terms.
	
	Define 
	\begin{equation}\label{I1}
	I_1:=\int_{\Lambda_h}\left(\partial_\nu v(\mathbf x)u_0(s\mathbf{x})-\partial_\nu u_0(s\mathbf{x})v(\mathbf{x})\right)\,d\sigma.
	\end{equation}
	Since $v\in H^2(D_\epsilon )$, following a similar argument in \cite[Proof of Theorem 2.1]{CDL2}, utilizing Cauchy-Schwarz inequality and the trace theorem, we can show that
	\begin{align}\label{I11}
	|I_1|&\leq \lVert\partial_\nu v\rVert_{L^2(\Lambda_h)}\lVert u_0\rVert_{L^2(\Lambda_h)}+\lVert\partial_\nu u_0\rVert_{L^2(\Lambda_h)}\lVert v\rVert_{L^2(\Lambda_h)}\notag\\
	&\leq\lVert v\rVert_{H^2(S_h)}(\lVert u_0\rVert_{L^2(\Lambda_h)}+\lVert\partial_\nu u_0\rVert_{L^2(\Lambda_h)})\leq Ce^{-c'\sqrt{s}}
	\end{align}
	as $s\rightarrow\infty$, where $C, c'$ are two positive constants.
	

By \eqref{eq:43 calpha}, 	one has that $v\in C^{\alpha}(\overline{S_h})$ for $\alpha\in(0,1)$. Thus $v$ has the following expansion
	\begin{equation}\label{v-expan}
	v(\mathbf x)=v(\mathbf{0})+\delta v(\mathbf x),\quad |\delta v(\mathbf x)|\leq \lVert v\rVert_{C^{\alpha}}|\mathbf{x}|^{\alpha},
	\end{equation}
	where $v(\mathbf{0})=0$ since $v=0$ on $\Gamma_h^\pm$. Hence we can derive by \eqref{v-expan} that
	\begin{equation}\notag
	\int_{S_h}k^2\omega_1vu_0(s\mathbf{x})\,d\mathbf{x}=k^2\omega_1\int_{S_h}(v(\mathbf{0})+\delta v(\mathbf x))u_0(s\mathbf x)\,d\mathbf x=k^2\omega_1\int_{S_h}\delta v(\mathbf x)u_0(s\mathbf x)\,d\mathbf x.
	\end{equation}
	
	Define 
	\begin{equation}\label{I2-new}
	I_2:=\int_{S_h}\delta v(\mathbf x)u_0(s\mathbf x)\,d\mathbf x,
	\end{equation}
	By \eqref{eq:xalpha}, there holds
	\begin{align}\label{I21-new}
	|I_2|
	&\leq\int_{S_h}|\delta v(\mathbf x)||u_0(s\mathbf x)|\,d\mathbf x\leq\lVert v\rVert_{C^{\alpha}}\int_{W}|u_0(s\mathbf x)||\mathbf x|^{\alpha}\,d \mathbf x\notag\\
	&\leq \frac{2(\theta_M-\theta_m )\Gamma(2\alpha+4) }{ \delta_W^{2\alpha+4}} \lVert v\rVert_{C^{\alpha}}s^{-\alpha-2}.
	\end{align}

	Define 
	\begin{equation}\label{I3-new}
	I_3^\pm:=\int_{\Gamma_h^\pm}(\eta_1-\eta_2)u_2^-(\mathbf x)u_0(s\mathbf{x})\,d\sigma.
	\end{equation}
	From \eqref{eq:43 calpha}, since $u_2^-\in C^{\alpha}(\overline{S_h})$ for $\alpha\in(0,1)$, it has the following expansion
	\begin{equation}\label{u2-expan}
	u_2^-(\mathbf x)=u_2^-(\mathbf{0})+\delta u_2^-(\mathbf{x}),\quad |\delta u_2^-(\mathbf x)|\leq \lVert u_2^-\rVert_{C^{\alpha}}|\mathbf x|^{\alpha}.
	\end{equation}
	Substituting \eqref{u2-expan} into $I_3^+$, we know that
	\begin{equation}\label{I31-new}
	I_3^+=(\eta_1-\eta_2)\int_{\Gamma_h^+}\left(u_2^-(\mathbf{0})+\delta u_2^-(\mathbf{x})\right)u_0(s\mathbf x)\,d\sigma:=(\eta_1-\eta_2)u_2^-(\mathbf{0})I_{31}^++(\eta_1-\eta_2)I_{32}^+,	
	\end{equation}
	where
	\begin{equation}\label{I31}
	I_{31}^+=\int_{\Gamma_h^+}u_0(s\mathbf x)\,d\sigma,
	\end{equation}
	and 
	\begin{equation}\label{I32}
	I_{32}^+=\int_{\Gamma_h^+}\delta u_2^-(\mathbf{x})u_0(s\mathbf x)\,d\sigma.
	\end{equation}
	
	Denote
	\begin{equation}\notag
	\omega(\theta)=-\cos\left(\frac{\theta}{2}+\pi\right), \quad \mu(\theta)=-\cos\left(\frac{\theta}{2}+\pi\right)-\bsi \sin\left(\frac{\theta}{2}+\pi\right). 
	\end{equation} 
	It is clear that $\omega(\theta)>0$ for $\theta_m<\theta<\theta_M$.
	Then by the variable substitution $t=\sqrt{sr}$, we have
	\begin{equation}\label{I311}
	I_{31}^+=\int_{0}^{h}e^{-\sqrt{sr}\mu(\theta_M)}\,dr=2s^{-1}\left(\mu(\theta_M)^{-2}-\mu(\theta_M)^{-2}e^{-\sqrt{sh}\mu(\theta_M)}-\mu(\theta_M)^{-1}\sqrt{sh}e^{-\sqrt{sh}\mu(\theta_M)}\right).
	\end{equation}
	Using \eqref{u2-expan}, it is easy to see that
	\begin{equation}\label{I322}
	|I_{32}^+|\leq \lVert u_2^-\rVert_{C^{\alpha}}\int_{0}^{h}r^\alpha e^{-\sqrt{sr}\omega(\theta_M)}\,dr=\Oh(s^{-\alpha-1}),
	\end{equation}
	by direct calculations.
	
	Using the similar argument of $I_3^+$ for $I_3^-$, we can derive the following equation
	\begin{equation}\label{I32-new}
	I_3^-=(\eta_1-\eta_2)\int_{\Gamma_h^-}\left(u_2^-(\mathbf{0})+\delta u_2^-(\mathbf{x})\right)u_0(s\mathbf x)\,d\sigma:=(\eta_1-\eta_2)u_2^-(\mathbf{0})I_{31}^-+(\eta_1-\eta_2)I_{32}^-,	
	\end{equation}
	where
	\begin{equation}\label{I31-}
	I_{31}^-=\int_{0}^{h}e^{-\sqrt{sr}\mu(\theta_m)}\,dr=2s^{-1}\left(\mu(\theta_m)^{-2}-\mu(\theta_m)^{-2}e^{-\sqrt{sh}\mu(\theta_m)}-\mu(\theta_m)^{-1}\sqrt{sh}e^{-\sqrt{sh}\mu(\theta_m)}\right),
	\end{equation}
	and 
	\begin{equation}\label{I32-}
	|I_{32}^-|\leq\lVert u_2^-\rVert_{C^{\alpha}}\int_{0}^{h}r^\alpha e^{-\sqrt{sr}\omega(\theta_m)}\,dr=\Oh(s^{-\alpha-1}).
	\end{equation}
	
	Moreover, by \eqref{u2-expan}, we have
	\begin{align}\label{id1}
	&\int_{S_h}k^2(\omega_2-\omega_1)u_2^-u_0(s\mathbf x)\,d\mathbf{x}=k^2(\omega_2-\omega_1)\int_{S_h}\left(u_2^-(\mathbf{0})+\delta u_2^-(\mathbf{x}))\right)u_0(s\mathbf{x})\,d\mathbf{x}\notag\\
	&=k^2(\omega_2-\omega_1)u_2^-(\mathbf 0)\int_{S_h}u_0(s\mathbf x)\,d\mathbf{x}+k^2(\omega_2-\omega_1)\int_{S_h}\delta u_2^-(\mathbf{x})u_0(s\mathbf x)\,d\mathbf{x}\notag\\
	&=k^2(\omega_2-\omega_1)u_2^-(\mathbf 0)\int_{W}u_0(s\mathbf x)\,d\mathbf{x}-k^2(\omega_2-\omega_1)u_2^-(\mathbf 0)\int_{W\backslash B_h}u_0(s\mathbf x)\,d\mathbf{x}\notag\\
	&+k^2(\omega_2-\omega_1)\int_{S_h}\delta u_2^-(\mathbf{x})u_0(s\mathbf x)\,d\mathbf{x}.
	\end{align}
	Define 
	\begin{equation}\label{I4}
	I_4:=\int_{W\backslash B_h}u_0(s\mathbf x)\,d\mathbf{x}\quad\mbox{and}\quad 
	I_5:=\int_{S_h}\delta u_2^-(\mathbf{x})u_0(s\mathbf x)\,d\mathbf{x}.
	\end{equation}
	We know that $I_4$ fulfills \eqref{eq:1.5}, and similar to $I_2$, $I_5$ satisfies
	\begin{equation}\label{I5}
	|I_5|\leq \frac{2(\theta_M-\theta_m )\Gamma(2\alpha+4) }{ \delta_W^{2\alpha+4}} \lVert u_2^-\rVert_{C^{\alpha}}s^{-\alpha-2}\quad\mbox{for}\quad\alpha\in(0,1).
	\end{equation}
	
	Therefore, we can deduce from the integral identity \eqref{sim1} that
	\begin{equation}\label{sim2}
	k^2(\omega_2-\omega_1)u_2^-(\mathbf 0)\int_{W}u_0(s\mathbf x)\,d\mathbf{x}= I_1+k^2\omega_1 I_2+I_3^\pm+k^2(\omega_2-\omega_1)u_2^-(\mathbf{0})I_4+k^2(\omega_1-\omega_2)I_5.
	\end{equation}
	Substituting \eqref{eq:u0w}, \eqref{I31-new} and \eqref{I32-new} into \eqref{sim2}, after rearranging terms, we have
	\begin{align}\label{sim3}
	&2(\eta_2-\eta_1)u_2^-(\mathbf{0})\left(\mu(\theta_M)^{-2}+\mu(\theta_m)^{-2}\right)s^{-1}=6\bsi k^2(\omega_1-\omega_2)u_2^-(\mathbf{0})(e^{-2\theta_M\bsi}-e^{-2\theta_m\bsi})s^{-2}\notag\\
	&+2(\eta_2-\eta_1)u_2^-(\mathbf{0})s^{-1}\big(\mu(\theta_M)^{-2}e^{-\sqrt{sh}\mu(\theta_M)}+\mu(\theta_M)^{-1}\sqrt{sh}e^{-\sqrt{sh}\mu(\theta_M)}\notag\\
	&+\mu(\theta_m)^{-2}e^{-\sqrt{sh}\mu(\theta_m)}+\mu(\theta_m)^{-1}\sqrt{sh}e^{-\sqrt{sh}\mu(\theta_m)}\big)+I_1+k^2\omega_1 I_2+(\eta_1-\eta_2)I_{32}^+\notag\\
	&+(\eta_1-\eta_2)I_{32}^-+k^2(\omega_2-\omega_1)u_2^-(\mathbf{0})I_4+k^2(\omega_1-\omega_2)I_5.
	\end{align}
	
	Combining with \eqref{I11}, \eqref{I21-new}, \eqref{I322}, \eqref{I32-}, \eqref{eq:1.5} and \eqref{I5}, multiplying $s$ on the both sides of \eqref{sim3} and taking limit as $s\rightarrow\infty$, we can obtain that 
	\begin{equation}\label{final2}
	2(\eta_2-\eta_1)u_2^-(\mathbf{0})\left(\mu(\theta_M)^{-2}+\mu(\theta_m)^{-2}\right)=0.
	\end{equation}
	Since 
	\begin{equation}\notag
	\mu(\theta_M)^{-2}+\mu(\theta_m)^{-2}=\frac{(\cos\theta_M+\cos\theta_m)+\bsi(\sin\theta_M+\sin\theta_m)}{(\cos\theta_M+\bsi\sin\theta_M)(\cos\theta_m+\bsi\sin\theta_m)},
	\end{equation}
	and $-\pi<\theta_m<\theta_M<\pi$ with $\theta_M-\theta_m\neq\pi$, we can deduce that $\cos\theta_M+\cos\theta_m$ and $\sin\theta_M+\sin\theta_m$ can not be zero simultaneously. This implies that 
	\begin{equation}\notag
	\mu(\theta_M)^{-2}+\mu(\theta_m)^{-2}\neq0.
	\end{equation}
	Due to  \eqref{11}, we know that 
	\begin{equation}\label{eq:eta1eta2}
			\eta_1=\eta_2,
	\end{equation}
which completes the proof of Case (a).

	\medskip
	
	 	\noindent {\bf Part 2.}~In this part, we prove that $\omega_1=\omega_2$.  Substituting \eqref{eq:eta1eta2} into \eqref{sim3}, it yields that {\bl
		\begin{equation}\label{new-id}
		6\bsi k^2(\omega_2-\omega_1)u_1^-(\mathbf 0)(e^{-2\theta_M \bsi }-e^{-2\theta_m \bsi }  ) s^{-2}=I_1+k^2\omega_1I_2+k^2(\omega_2-\omega_1)u_2^-(\mathbf{0})I_4+k^2(\omega_1-\omega_2)I_5,
	\end{equation}  }
	where $I_1, I_2, I_3$ and $I_4$ fulfill \eqref{I11}, \eqref{I21-new}, \eqref{eq:1.5} and \eqref{I5}, respectively.
	
	Multiplying $s^2$ on the both sides of \eqref{new-id} and letting $s\rightarrow\infty$, we can prove that 
	\begin{equation}\label{final1}
			6\bsi k^2(\omega_2-\omega_1)u_1^-(\mathbf 0)(e^{-2\theta_M \bsi }-e^{-2\theta_m \bsi }  )=0.
	\end{equation}
	Since $-\pi<\theta_m<\theta_M<\pi$, it indicates that $e^{-2\theta_M \bsi }-e^{-2\theta_m \bsi }\neq 0$. From \eqref{11},  we can obtain directly that
	\begin{equation}\notag
		u_1^-(\mathbf{0})=u_2^- (\mathbf{0})\neq 0.
	\end{equation}
Therefore, by virtue of \eqref{final1}, we have  $\omega_1= \omega_2$.

\medskip
	\noindent {\bf Part 3.}~	Finally, we prove \eqref{eq:lem41 u1-}. Indeed, from \eqref{11}, combining with the facts that $\eta_1=\eta_2$ (Part 1) and the conductive boundary conditions in \eqref{eq:lemma41},  by the trace theroem one has
	\begin{equation}\label{eq:trans}
		u_1^-=u_1^+=u_2^+=u_2^-,\quad \partial_\nu u_1^+=\partial_\nu u_2^+,\quad \partial_\nu u_1^-=\partial_\nu u_2^-,\quad\mbox{on }\Gamma_h^\pm,
	\end{equation}	
	Therefore, since $\omega_1=\omega_2$ (Part 2.), by virtue of \eqref{eq:trans}, it is clear that
	\begin{equation}\label{main1-v}
		\begin{cases}
		\Delta v+k^2 \omega_2v=0&\mbox{in }S_h,\\
		v=\partial_\nu v=0&\mbox{on }\Gamma_h^\pm  ,
		\end{cases}
	\end{equation}
	where $v:=u_1^--u_2^-$. Using Holmgren's uniqueness principle (cf. \cite{CK}), we finish the proof of this lemma. 
	\end{proof}


Now, we are in a position to present the detailed  proofs of Theorem \ref{main1} and Theorem \ref{main2}.

\begin{proof}[Proof of Theorem \ref{main1}]
From the statement of Theorem \ref{main1},   the shape of $\Omega$ and the corresponding polygonal-cell partitions  $\Sigma_\ell$, $\ell=1,\ldots,N$, are known in advance. Then we prove this theorem by contradiction.  Suppose that $\lambda_{1}^*\neq \lambda_{2}^*$ or there exits an index $1\leq \ell_0\leq N$ such that 
$
q_{\ell_0,1} \neq q_{\ell_0,2} .
$  
From (3) in Definition \ref{def:2}, we know that   $\Sigma_{\ell_0}$ has a vertex $\mathbf{x}_c$ intersected by two adjacent edges $\Gamma^\pm$ of $\Omega$.
Since $-\Delta$ is invariant under rigid motion, without loss of generality, we assume that $\mathbf{x}_c=\mathbf{0}$ and for sufficiently small $h>0$,  there exists $B_h(\mathbf{0})$ centered at $\mathbf{0}$ with radius $h$, such that $S_h=\Omega\cap B_h(\mathbf{0})$ and $\Gamma_h^\pm=\partial\Omega\cap B_h(\mathbf{0})$, $\Gamma_h^\pm  \subset \partial {\mathbf G}$, where ${\mathbf G}=\mathbb {R}^2 \backslash \overline{ \Omega}$, $S_h$ and $\Gamma_h^\pm$ are the same as \eqref{nota2};
see Figure \ref{fig:main-proof} for the schematic illustration.
	
	\begin{figure}[htbp]
		\centering
		\includegraphics[width=0.5\linewidth]{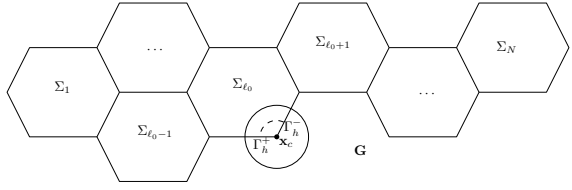}
		\caption{Schematic illustration of the cell geometry.}
		\label{fig:main-proof}
	\end{figure}

Let $u_1$ and $u_2$ be the total wave {\bl fields} associated with the admissible polygonal-cell conductive medium body  $(\Omega;q_j,\lambda_j )$, $j=1,2$.  Since $u_1^{\infty}(\hat{\mathbf{x}}; u^i)=u_2^\infty(\hat{\mathbf{x}}; u^i)$ for all $\hat{\mathbf{x}}\in\mathbb{S}^1$, by Rellich's Theorem (cf. \cite{CK}) and the unique continuation principle, we can derive that 
	\begin{equation}\label{main1-1}
		u_1^+=u_2^+\quad\mbox{in } B_h \backslash \overline{ S_h} \subset \mathbf{G}  ,
	\end{equation}
	and also
	\begin{equation}\label{main1-2}
	\begin{cases}
		u^+_1= u^-_1,\quad \partial_\nu u^+_1 + \lambda_{1}^* \ u^+_1=\partial_\nu u^- _1& \mbox{ on }\  \Gamma_h^\pm , \\[5pt]
u^+_2= u^-_2,\quad \partial_\nu u^+_2 + \lambda_{2}^* \ u^+_2=\partial_\nu u^- _2& \mbox{ on }\  \Gamma_h^\pm , 
\end{cases}
	\end{equation}
	by combining with the conductive transmission boundary conditions in \eqref{eq:model1}. Furthermore, it yields that
	\begin{equation}\label{eq:cell pde 436}
	\begin{cases}
		\Delta u^-_1+k^2 q_{\ell_0,1}  u^-_1=0 & \mbox{ in }\ S_h, \\[5pt] 
\Delta u^+_1 +k^2   u^+_1=0 & \mbox{ in }\ B_h \backslash \overline{ S_h}, \\[5pt] 
\Delta u^-_2+k^2 q_{\ell_0,2}  u^-_2=0 & \mbox{ in }\ S_h, \\[5pt] 
\Delta u^+_2 +k^2  u^+_2=0 & \mbox{ in }\ B_h \backslash \overline{ S_h}. 
	\end{cases}
	\end{equation}
 
 Recall that  $\Omega$ is an admissible polygonal-cell conductive medium body. Using the fact that $u^+_1,\, u_2^+ \in H^2( B_h \backslash \overline{ S_h} )$, one claims that $u^+_1(\mathbf{0})=u_2^+(\mathbf{0})\neq 0$.  In view of \eqref{main1-1}, \eqref{main1-2} and \eqref{eq:cell pde 436}, with the help of Lemma \ref{prop1}, we have 
 \[
 q_{\ell_0,1} = q_{\ell_0,2} \mbox{ and } \lambda_{1}^* = \lambda_{2}^*, 
 \]
 where we arrive at a contradiction. 
 
 The proof is complete. 
\end{proof}

\begin{proof}[Proof of Theorem \ref{main2}]
We prove this theorem by mathematical induction. 
From Theorem \ref{Th:3.1}, one has $\partial \Omega_1=\partial \Omega_2$,  which indicates that $\partial \Sigma_{1,1}=\partial \Sigma_{1,2}$. Using the similar argument in proving Theorem \ref{main1}, it can be shown that $q_{1,1}=q_{1,2}$ and $\lambda_{1,1}=\lambda_{1,2}$. Assume that there exits an index $n \in \mathbb{N} \backslash\{1\} $ such that 
\begin{equation}
	\partial \Sigma_{\ell,1}=\partial \Sigma_{\ell,2}, \quad q_{\ell,1}=q_{\ell,2},  \quad \lambda_{\ell,1}=\lambda_{\ell,2},\quad {\bl \ell=2,\ldots,n-1,  }
\end{equation}
where $\Sigma_{\ell,j}$ are  polygonal-nest partitions of $\Omega_j$,   
$j=1,2$. With the help of Lemma \ref{prop1}, hence we can recursively prove that
\begin{equation}\label{eq:u1l+1}
	u_{\ell,1}=u_{\ell,2} \mbox{ in } U_\ell=\Sigma_\ell\backslash \overline{\Sigma _{\ell+1} }, \quad \ell=1,2,\ldots, n-1
\end{equation}
by using $u_1^{\infty}(\hat{\mathbf{x}}; u^i)=u_2^\infty(\hat{\mathbf{x}}; u^i)$ and \eqref{eq:lem41 u1-}, where   $u_{\ell,1}=u_1\big |_{U_\ell }$ and  $u_{\ell,2}=u_2\big |_{U_\ell }$, with $u_1$ and $u_2$ being the total wave {fields} associated with the admissible polygonal-nest  conductive medium bodies  $(\Omega_j; q_j,\lambda_j)$, $j=1,2$. We divide the proof into two parts. 

\medskip

\noindent {\bf Part 1.}~We first show that $\partial \Sigma_{n,1}=\partial \Sigma_{n,2}$. From \eqref{eq:u1l+1}, using the similar argument in proving Theorem \ref{Th:3.1}, by Lemma \ref{Th:1.2}, one can prove $\partial \Sigma_{n,1}=\partial \Sigma_{n,2}$ directly by contradiction. 

\medskip

\noindent {\bf Part 2.}~We prove that
\begin{equation}\label{eq:449}
	\quad q_{n,1}=q_{n,2},  \quad \lambda_{n,1}=\lambda_{n,2}. 
\end{equation}
Consider a vertex $\mathbf{x}_c$ of $\Sigma_{n,1}$. Since $-\Delta$ is invariant under rigid motion, without loss of generality, we may assume that $\mathbf{x}_c=\mathbf 0$. Recall that  $u_1$ and $u_2$ are the total wave {\bl fields} associated with the admissible polygonal-nest conductive medium bodies   $(\Omega_j;q_j,\lambda_j )$, $j=1,2$, respectively. For sufficiently small $h$, suppose that  $S_h  \Subset U_{n,1}=\Sigma_{n,1} \backslash \overline{\Sigma_{n+1,1} } $;  see Figure \ref{fig:main2-proof} for the schematic illustration for instance.
	\begin{figure}[htbp]
		\centering
		\includegraphics[width=0.35\linewidth]{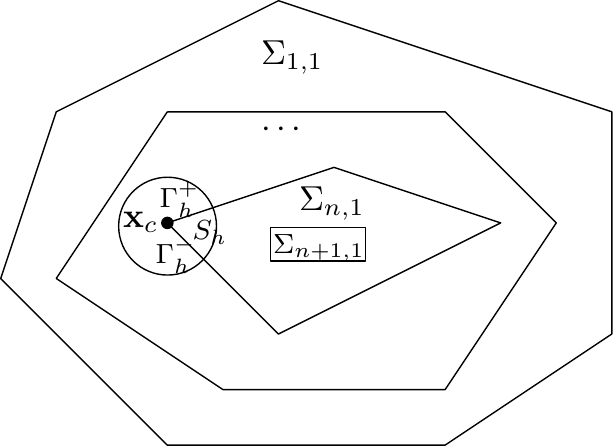}
		\caption{Schematic illustration of the nest geometry.}
		\label{fig:main2-proof}
	\end{figure}
Therefore, it yields that
\begin{equation}\label{eq:450}
\begin{cases}
\Delta u_{n,1}^-+k^2 q_{n,1} u_{n,1}^-=0 & \mbox{ in }\ S_h, \\[5pt] 
\Delta u_{n-1,1}^+ +k^2  q_{n-1,1} u_{n-1,1}^+=0 & \mbox{ in }\ B_h \backslash \overline{ S_h}, \\[5pt] 
\Delta u_{n,2}^-+k^2 q_{n,2} u_{n,2}^-=0 & \mbox{ in }\ S_h, \\[5pt] 
\Delta u_{n-1,2}^+ +k^2  q_{n-1,2} u_{n-1,2}^+=0 & \mbox{ in }\ B_h \backslash \overline{ S_h}, \\[5pt] 
u_{n,1}^-= u_{n-1,1}^+,\quad \partial_\nu u^+_{n-1,1} + \lambda_{n,1} u^+_{n-1,1}=\partial_\nu u^- _{n,1}& \mbox{ on }\  \Gamma_h^\pm , \\[5pt]
u^-_{n,2}= u^+_{n-1,2},\quad \partial_\nu u^+_{n-1,2} + \lambda_{n,2} u^+_{n-1,2}=\partial_\nu u^- _{n,2}& \mbox{ on }\  \Gamma_h^\pm .
\end{cases}
\end{equation}
 It is obvious that $u^+_{n-1,1}=u^+_{n-1,2}$ by the virtue of \eqref{eq:u1l+1}. In view of \eqref{eq:450}, with the help of Lemma \ref{prop1}, we can prove \eqref{eq:449}. 
 
 We can prove $N_1=N_2$ by using the contradiction. Indeed, if $N_1\neq  N_2$, without loss of generality, we assume that $N_1<N_2$. Therefore,  there exits a corner $\Sigma_{N_2+1,2}$ lying inside of  $\Sigma_{N_1,1}$. From Lemma \ref{Th:1.2}, we can prove that the total wave filed vanish at the vertex of this corner, which contradicts to the admissible condition of a polygonal-nest conductive medium body. 

The proof is complete. 
\end{proof}

\section{Concluding remarks} \label{sect:6}

As mentioned several times before, we would like to remark that the unique identifiability results in Theorems~\ref{main1} and \ref{main2} for the inverse problem \eqref{inverse11} associated with the scattering system \eqref{eq:model1} can be extended to the inverse problem \eqref{inverse22} associated with the scattering system \eqref{eq:model4}. 
In fact, it is easily seen from Section~\ref{sec3} that all the arguments in establishing the unique identifiability results are localized around a corner of the concerned conductive medium cell, and they are independent of the incident wave $u^i$ in \eqref{eq:model1}. The only place that we made use of the incident wave is in Definition~\ref{admissible} about the admissibility condition of the scattering medium body. Hence, as long as the admissibility condition in Definition~\ref{admissible} and the support of the source $\widetilde\psi$ stays a positive distance away from the conductive interface, then all the results derived in this paper for the inverse problem \eqref{inverse11} can be extended to the inverse problem \eqref{inverse22}. We would like to point out that one needs not to know $\widetilde\psi$ in advance, and it may be even located and supported away from the underlying conductive medium body. 

We would also like to remark that the arguments developed in this paper can be refined to establish certain unique identifiability results for more general conductive medium bodies than those considered in the present paper. However, it would be more interesting to consider the corresponding inverse problems in three dimensions associated with the Maxwell systems \eqref{eq:sca1} and \eqref{eq:model3}, which shall be a subject of our future study.

\section*{Acknowledgement}


The work of H Liu was supported by a startup fund from City University of Hong Kong and the Hong Kong RGC general research funds (projects 12302017, 12302018 and 12302919).

\end{document}